\documentclass{gtmon_a}
\pdfoutput=1
\usepackage{pinlabel}

\proceedingstitle{Heegaard splittings of 3--manifolds (Haifa 2005)}
\conferencestart{10 July 2005}
\conferenceend{19 July 2005}
\conferencename{Heegaard splittings of 3--manifolds}
\conferencelocation{Haifa}

\editor{Cameron Gordon}
\givenname{Cameron}
\surname{Gordon}

\editor{Yoav Moriah}
\givenname{Yoav}
\surname{Moriah}

\title{Waldhausen's Theorem}

\author{Saul Schleimer}
\givenname{Saul}
\surname{Schleimer}
\address{Department of Mathematics\\
        University of Warwick\\\newline
        Coventry, CV4 7AL\\UK}
\email{s.schleimer@warwick.ac.uk}

\volumenumber{12}
\issuenumber{}
\publicationyear{2007}
\papernumber{12}
\startpage{299}
\endpage{317}

\doi{}
\MR{}
\Zbl{}

\arxivreference{} 

\keyword{triangulations of three-manifolds}
\keyword{Heegaard splittings}
\keyword{stabilization}
\keyword{three-sphere}
\subject{primary}{msc2000}{57M40}
\subject{primary}{msc2000}{57N10}
\subject{secondary}{msc2000}{57Q15}

\received{16 January 2007}
\revised{1 October 2007}
\accepted{6 November 2007}
\published{3 December 2007}
\publishedonline{3 December 2007}
\proposed{}
\seconded{}
\corresponding{}
\version{}

\renewenvironment{quote}{\par\leftskip25pt\rightskip25pt\it}
{\par\leftskip0pt\rightskip0pt}

\newcommand{\calK}{\mathcal{K}}

\newcommand{\calT}{\mathcal{T}}

\newcommand{\CC}{\mathbb{C}}

\newcommand{\NN}{\mathbb{N}}

\newcommand{\RR}{\mathbb{R}}

\renewcommand{\setminus}{{\smallsetminus}}

\newcommand{\st}{\mathbin{\mid}}
\newcommand{\from}{\co}

\newcommand{\cross}{{\times}}

\newcommand{\bdy}{\partial} 

\newcommand{\genus}{\operatorname{genus}}
\newcommand{\interior}{{\operatorname{interior}}}
\newcommand{\connect}{\#} 
\newcommand{\RRPP}{\mathbb{RP}} 

\newcommand{\refsec}[1]{\fullref{Sec:#1}}
\newcommand{\refthm}[1]{\fullref{Thm:#1}}

\newcommand{\reflem}[1]{\fullref{Lem:#1}}

\newcommand{\refrem}[1]{\fullref{Rem:#1}}

\newcommand{\refex}[1]{\fullref{Ex:#1}}
\newcommand{\reffig}[1]{\fullref{Fig:#1}}

\renewcommand{\equiv}{\approx}

\makeatletter
\def\cnewtheorem#1[#2]#3{\newtheorem{#1}{#3}[section]
\expandafter\let\csname c@#1\endcsname\c@theorem}

\theoremstyle{plain}
\newtheorem{theorem}{Theorem}[section]
\cnewtheorem{corollary}[theorem]{Corollary}
\cnewtheorem{lemma}[theorem]{Lemma}
\cnewtheorem{conjecture}[theorem]{Conjecture}
\cnewtheorem{proposition}[theorem]{Proposition}
\cnewtheorem{principle}[theorem]{Principle}
\cnewtheorem{axiom}[theorem]{Axiom}
\newtheorem*{WTheorem}{\fullref{Thm:Waldhausen}}

\theoremstyle{definition}
\cnewtheorem{definition}[theorem]{Definition}
\cnewtheorem{remark}[theorem]{Remark}
\newtheorem*{remark*}{Remark}
\cnewtheorem{claim}[theorem]{Claim}
\newtheorem*{claim*}{Claim}
\cnewtheorem{problem}[theorem]{Problem}
\cnewtheorem{solution}[theorem]{Solution}
\cnewtheorem{remarks}[theorem]{Remarks}
\cnewtheorem{exercise}[theorem]{Exercise}
\cnewtheorem{example}[theorem]{Example}
\cnewtheorem{fact}[theorem]{Fact}
\cnewtheorem{question}[theorem]{Question}
\newtheorem*{question*}{Question}
\newtheorem*{answer*}{Answer}
\newtheorem*{application*}{Application}
\cnewtheorem{algorithm}[theorem]{Algorithm}

\makeatother

\begin{document}

\begin{abstract}
This note is an exposition of Waldhausen's proof of Waldhausen's
Theorem: the three-sphere has a single Heegaard splitting, up to
isotopy, in every genus.  As a necessary step we also give a sketch of
the Reidemeister--Singer Theorem.
\end{abstract}

\maketitle

\section{Introduction}
\label{Sec:Intro}

Waldhausen's Theorem~\cite{Waldhausen68} tells us that Heegaard
splittings of the three-sphere are unique up to isotopy.  This is an
important tool in low-dimensional topology and there are several
modern proofs (Jaco and Rubinstein \cite{JacoRubinstein06}, Rieck
\cite{Rieck06} Rubinstein and Scharlemann
\cite{RubinsteinScharlemann96} and Scharlemann and Thompson
\cite{ScharlemannThompson94a}).
Additionally, at least two survey articles on Heegaard splittings
(Johnson \cite{Johnson07} and Scharlemann \cite{Scharlemann02}) include
proofs of Waldhausen's Theorem.

This note is intended as an exposition of Waldhausen's original proof,
as his techniques are still of interest.  See, for example, Bartolini
and Rubinstein's~\cite{BartoliniRubinstein06} classification of
one-sided splittings of $\RRPP^3$.

In \refsec{Foundations} we recall foundational material, set out the
necessary definitions and give a precise statement of Waldhausen's
Theorem.  \refsec{Stable} is devoted to {\em stable equivalence} of
splittings and a proof of the Reidemeister--Singer Theorem.  In
\refsec{Meridian} we discuss Waldhausen's {\em good} and {\em great}
systems of meridian disks.  \refsec{Proof} gives the proof of
Waldhausen's Theorem.  Finally, \refsec{Remarks} is a brief account of
the work-to-date on the questions raised by Waldhausen in Section~4 of
his paper.  

\subsection*{Acknowledgements} 

The material in this paper's Sections~\ref{Sec:Meridian}
and~\ref{Sec:Proof} do not pretend to any originality: often I simply
translate Waldhausen's paper directly.  I also closely follow, in
places word for word, a copy of handwritten and typed notes from the a
seminar on Waldhausen's proof, held at the University of Wisconsin at
Madison, Fall 1967.  The notes were written by Ric Ancel, Russ
McMillan and Jonathan Simon.  The speakers at the seminar included
Ancel and Simon.  

I thank Professors Ancel and Simon for permitting me to use these
seminar notes.  I thank Cameron Gordon and Yoav Moriah for their
encouraging words and patience during the writing of this paper.
Also, I am grateful to the referee for many helpful comments.

This work is in the public domain.

\section{Foundations}
\label{Sec:Foundations}

The books by Hempel~\cite{Hempel76} and Rolfsen~\cite{Rolfsen76} and
also Hatcher's notes \cite{Hatcher01} are excellent references on
three-manifolds.  Moise's book \cite{Moise77} additionally covers
foundational issues in PL topology, as does the book by Rourke and
Sanderson~\cite{RourkeSanderson72}.

We will use $M$ to represent a connected compact orientable
three-manifold.  We say $M$ is {\em closed} if the boundary $\bdy M$
is empty.  A {\em triangulation} of $M$ is a simplicial complex
$\calK$ so that the underlying space $||\calK||$ is homeomorphic to
$M$.  When no confusion can arise, we will regard the cells of
$||\calK||$ as being subsets of $M$.

\begin{example}
\label{Exa:ThreeSphere}
The three-sphere is given by
$$
S^3 = \{ (z, w) \in \CC^2 \st |z|^2 + |w|^2 = 2 \}.
$$ 
The boundary of the four-simplex gives a five-tetrahedron
triangulation of $S^3$.
\end{example}

\noindent
Requiring that $M$ be given with a triangulation is not a restriction:

\begin{theorem}[Triangulation]
\label{Thm:Triangulation}
Every compact three-manifold $M$ admits a trian\-gulation.\qed
\end{theorem}

Furthermore, in dimension three there is only one PL structure:

\begin{theorem}[Hauptvermutung]
\label{Thm:Hauptvermutung}
Any two triangulations of $M$ are related by a PL homeomorphism that
is isotopic to the identity in $M$.  \qed
\end{theorem}

These theorems are due to Moise~\cite{Moise52}.  An alternative proof
is given by Bing~\cite{Bing59}.  Our version of the Hauptvermutung may
be found in Hamilton~\cite{Hamilton76}.


We now return to notational issues.  We will use $F$ to represent a
closed connected orientable surface embedded in $M$.  A simple closed
curve $\alpha \subset F$ is {\em essential} if $\alpha$ does not bound
a disk in $F$.

For any $X \subset M$ we use $U(X)$ to denote a regular open
neighborhood of $X$, taken in $M$.  This neighborhood is assumed to be
small with respect to everything relevant.  If $X$ is a topological
space, we use $|X|$ to denote the number of components of $X$.

A {\em handlebody}, usually denoted by $V$ or $W$, is a homeomorph of
a closed regular neighborhood of a finite, connected graph embedded in
$\RR^3$.  The {\em genus} of $V$ agrees with the genus of $\bdy V$.
Notice that if $\calK$ is a triangulation of $M$ then a closed regular
neighborhood of the one-skeleton of $||\calK||$ is a handlebody
embedded in $M$.

A disk $v_0$ is {\em properly embedded} in a handlebody $V$ if $v_0
\cap \bdy V = \bdy v_0$; this definition generalizes naturally to
surfaces and arcs contained in bounded three-manifolds and also to
arcs contained in bounded surfaces.

A {\em Heegaard splitting} is a pair $(M, F)$ where $M$ is a closed
oriented three-manifold, $F$ is an oriented closed surface embedded in
$M$, and $M \setminus U(F)$ is a disjoint union of handlebodies.

\begin{example}
\label{Exa:GenusZero}
There is an equatorial two-sphere $S^2 \subset S^3$:
$$
S^2 = \{ (z, w) \in S^3 \st {\rm Im}(w) = 0 \}.
$$ 
Note that $S^2$ bounds a three-ball on each side.  We call $(S^3,
S^2)$ the {\em standard} splitting of genus zero.
\end{example}

The Alexander trick proves that any three-manifold with a splitting of
genus zero is homeomorphic to the three-sphere.
Furthermore, we have:

\begin{theorem}[Alexander~\cite{Alexander24}]
\label{Thm:Alexander}
Every PL two-sphere in $S^3$ bounds three-balls on both sides.  \qed
\end{theorem}

\noindent
See Hatcher~\cite{Hatcher01} for a detailed proof.  It follows that every PL
two-sphere gives a Heegaard splitting of $S^3$.

\begin{example}
\label{Exa:GenusOne}
There is a torus $T \subset S^3$:
$$
T = \{ (z, w) \in S^3 \st |z| = |w| = 1 \}.
$$ 
It is an exercise to check that $T$ bounds a solid torus ($D^2 \cross
S^1$) on each side.  We call $(S^3, T)$ the standard splitting of
$S^3$ of genus one.  
\end{example}


The three-manifolds admitting splittings of genus one are $S^3$, $S^2
\cross S^1$ and the lens spaces.  As an easy exercise from the
definitions we have:

\begin{lemma}
\label{Lem:Neighborhood}
Suppose $\calK$ is a triangulation of a closed orientable manifold
$M$.  Suppose that $F$ is the boundary of a closed regular
neighborhood of the one-skeleton of $||\calK||$.  Then $(M, F)$ is a
Heegaard splitting.  \qed
\end{lemma}

See, for example, Rolfsen \cite[page 241]{Rolfsen76}.  The splitting $(M,
F)$ so given is the splitting {\em associated} to the triangulation
$\calK$.  As an immediate consequence of the Triangulation Theorem
(\ref{Thm:Triangulation}) and \reflem{Neighborhood} we find that every
closed three-manifold has infinitely many Heegaard splittings.  To
control this extravagance of examples we make:

\begin{definition}
\label{Def:Equiv}
A pair of Heegaard splittings $(M, F)$ and $(M, F')$ are {\em
equivalent}, written $(M, F) \equiv (M, F')$, if there is a
homeomorphism $h \from M \to M$ such that
\begin{itemize}
\item
$h$ is isotopic to the identity and
\item
$h|F$ is an orientation preserving homeomorphism from $F$ to $F'$. 
\end{itemize}
\end{definition}

It is an important visualization exercise to show that $(S^3, T)$ is
equivalent to $\left( S^3, -T \right)$.  Here $-T$ is the torus $T$
equipped with the opposite orientation.  We now have another
foundational theorem:

\begin{theorem}[Gugenheim~\cite{Gugenheim53}]
\label{Thm:IsotopyOfBalls}
If $B$ and $B'$ are PL three-balls in a three-manifold $M$ then there
is an isotopy of $M$ carrying $B$ to $B'$.  \qed
\end{theorem}

See Theorem 3.34 of Rourke and Sanderson~\cite{RourkeSanderson72} for
a discussion.  They also give as Theorem~4.20 a relative version. In
any case, it follows that all genus zero splittings of $S^3$ are
equivalent to the standard one, so justifying the name.

\begin{exercise}
\label{Ex:GenusOne}
Show that any genus one splitting of $S^3$ is isotopic to the standard
one. (Corollary~4.16 of~\cite{RourkeSanderson72} may be useful.)
\end{exercise}

Waldhausen's Theorem generalizes this result to every genus:

\begin{WTheorem}
If $F$ and $F'$ are Heegaard splittings of $S^3$ of the same genus
then $(S^3, F)$ is equivalent to $(S^3, F')$.
\end{WTheorem}

\begin{remark}
Waldhausen's original statement is even simpler: 
\begin{quote}
Wir zeigen, da\ss{} es nur die bekannten gibt.
\end{quote}
That is: ``We show that only the well-known [splittings of $S^3$]
exist.''
\end{remark}


\section{Stabilization and the Reidemeister--Singer Theorem}
\label{Sec:Stable}

A key step in Waldhausen's proof is the Reidemeister--Singer Theorem
(\fullref{Thm:RS}, below).  In this section we lay out the
necessary definitions and sketch a proof of the Reidemeister--Singer
Theorem.  Most approaches to Reidemeister--Singer, including ours, are
via piecewise linear topology.  Bonahon in an unpublished manuscript
has given a proof relying on Morse theory.

For further details and the history of the problem we refer the reader
to the original papers of Reidemeister and
Singer~\cite{Reidemeister33,Singer33} as well as the more modern
treatment by Craggs~\cite{Craggs76}.  A version of Craggs' proof is
also given by Fomenko \cite[Theorem~5.2]{FomenkoMatveev97}.  Note also that
Lei~\cite{Lei00}, in an amusing reversal, gives a very short proof of
the Reidemeister--Singer Theorem by assuming Waldhausen's Theorem.

We begin by stating the basic definitions and then the theorem.

\begin{definition}
\label{Def:UnknottedArc}
Suppose that $V$ is a handlebody.  A properly embedded arc $\alpha
\subset V$ is {\em unknotted} if there is an arc $\beta \subset \bdy
V$ and an embedded disk $B \subset V$ so that $\bdy \alpha = \bdy
\beta$ and $\bdy B = \alpha \cup \beta$.
\end{definition}

\begin{definition}
\label{Def:StabilizeOne}
Suppose that $(M, F)$ is a Heegaard splitting with handlebodies $V$
and $W$.  Let $\alpha$ be an unknotted arc in $V$.  Let $F' = \bdy (V
\setminus U(\alpha)) = (V \setminus U(\alpha)) \cap (W \cup
\overline{U(\alpha)})$.  Then the pair $(M, F')$ is a {\em
stabilization} of $F$ in $M$.  Also, the pair $(M, F)$ is a {\em
destabilization} of $(M, F')$.
\end{definition}

Observe that $(S^3, T)$ is isotopic to a stabilization of $(S^3,
S^2)$.  It is an exercise to prove, using the relative version of
\refthm{IsotopyOfBalls} and \refex{GenusOne},
that if $(M, F')$ and $(M, F'')$ are stabilizations of $(M, F)$, then
$(M, F') \equiv (M, F'')$.  On the other hand, as discussed below,
destabilization need {\em not} be a unique operation.


Recall that the {\em connect sum} $M \connect N$ is obtained by
removing the interior of a ball from each of $M$ and $N$ and then
identifying the resulting boundary components via an orientation
reversal.

\begin{definition}
\label{Def:StabilizeTwo}
Let $(M, F)$ be a Heegaard splitting.  Let $(S^3, T)$ be the standard
genus one splitting of $S^3$.  Pick embedded three-balls meeting $F
\subset M$ and $T \subset S^3$ in disks.  The {\em connect sum} of the
splittings is the connect sum of pairs: $(M, F) \connect (S^3, T) = (M
\connect S^3, F \connect T)$.
\end{definition}

Again, this operation is unique and the proof is similar to that of
uniqueness of stabilization.  This is not a surprise, as stabilization
and connect sum with $(S^3, T)$ produce equivalent splittings.  Thus
we do not distinguish between them notationally.

\begin{remark}
\label{Rem:Forest}
Fix a manifold $M$.  We may construct a graph $\Sigma(M)$ where
vertices are equivalence classes of splittings and edges correspond to
stabilizations.  From \refthm{Triangulation} it follows that
$\Sigma(M)$ is nonempty.  The uniqueness of stabilization implies that
$\Sigma(M)$ has no cycles and so is a forest.  Finally, $\Sigma(M)$ is
infinite because splittings of differing genera cannot be isotopic.
\end{remark}

Define $(M, F) \connect_n (S^3, T) = ((M, F) \connect_{n-1}(S^3, T))
\connect (S^3, T)$.

\begin{definition}
\label{Def:StabEquiv}
Two splittings, $(M, F)$ and $(M, F')$, are {\em stably equivalent} if
there are $m, n \in \NN$ so that $(M, F) \connect_m (S^3, T) \equiv
(M, F') \connect_n (S^3, T)$.
\end{definition}

We now may state:

\begin{theorem}[Reidemeister--Singer]
\label{Thm:RS}
Suppose that $M$ is a closed, connected, orientable three-manifold.
Then any two Heegaard splittings of $M$ are stably equivalent.
\end{theorem}

\begin{remark}
\label{Rem:Tree}
The theorem may be restated as follows: $\Sigma(M)$ is connected.
Since \refrem{Forest} shows that $\Sigma(M)$ is a forest, it is a
tree.
\end{remark}

We say that $(M, F)$ is {\em unstabilized} if it is not equivalent to
a stabilized splitting.  Waldhausen calls such splittings ``minimal''.
However modern authors reserve ``minimal'' to mean {\em minimal
genus}.  This is because there are manifolds containing unstabilized
splittings that are not of minimal genus.  For examples, see
Sedgwick's discussion of splittings of Seifert fibered
spaces~\cite{Sedgwick99}.  Note that unstabilized splittings
correspond to leaves of the tree $\Sigma(M)$.

Finally, there are fixed manifolds that contain unstabilized
splittings of arbitrarily large genus.  The first such examples are
due to Casson and Gordon~\cite{CassonGordon85}.  The papers of Kobayashi
\cite{Kobayashi92}, Lustig and Moriah \cite{LustigMoriah00} and Moriah
\cite{MoriahEtAl06} contain generalizations.

We now set out the tools necessary for our proof of \refthm{RS}.  A
{\em pseudo-triangulation} $\calT = \{ \Delta_i \}$ of a
three-manifold $M$ is a collection of tetrahedra together with face
identifications.  We require that the resulting quotient space
$||\calT||$ be homeomorphic to $M$ and that every open cell of $\calT$
embeds.  We do not require that $\calT$ be a simplicial complex.  It
is a pleasant exercise to find all pseudo-triangulations of $S^3$
consisting of a single tetrahedron.

As with triangulations, if $\calT$ is a pseudo-triangulation of $M$
then the boundary of a closed regular neighborhood of the one-skeleton
of $||\calT||$ is a Heegaard splitting of $M$.  Notice that the second
barycentric subdivision of $\calT$ is a triangulation of $M$.

\begin{lemma}
\label{Lem:ExistsATriangulation}
For any splitting $(M, F)$ there is an $n \in \NN$ and a triangulation
$\calK$ of $M$ so that $(M, F) \connect_n (S^3, T)$ is associated to
$\calK$.  
\end{lemma}


\begin{proof}
We may assume, stabilizing if necessary, that $F$ has genus at least
one.  Now, $F$ cuts $M$ into a pair of handlebodies $V$ and $W$, both
of genus $g$.  Choose $g$ disks $\{ v_i \}$ properly embedded in $V$
so that the $v_i$ cut $V$ into a ball.  Choose $\{ w_j \}$ in $W$
similarly.  After a proper isotopy of the $v_i$ inside of $V$ we may
assume that all components of $F \setminus \Gamma$ are disks.  Here
$\Gamma = F \cap (( \cup v_i ) \cup ( \cup w_j ))$ is the {\em
Heegaard diagram} of $(F, v_i, w_j)$.

We build a pseudo-triangulation $\calT$ of $M$, with exactly two
vertices, by taking the dual of the two-complex $F \cup ( \cup v_i )
\cup ( \cup w_j )$.  It follows that $\calT$ has a tetrahedron for
every vertex of $\Gamma$, a face for every edge of $\Gamma$, an edge
for every face of $\Gamma$, an edge for each of the $2g$ disks and
exactly two vertices.

Let $\calT_V$ be the union of the edges of $\calT$ dual to the disks
$v_i$. Define $\calT_W$ similarly.  Let $e$ be any edge of $\calT$
connecting the two vertices of $\calT^0$.  Notice that $F$ is isotopic
to the boundary of a regular neighborhood of $\calT_V$.  After $g$
stabilizations of $F$ we obtain a surface $F'$ that is isotopic to the
boundary of a regular neighborhood of $\calT_V \cup e \cup \calT_W$.
Now a further sequence of stabilizations of $F'$ gives the splitting
associated to $\calT$.  We end with an easy exercise: if a splitting
$(M, G)$ is associated to a pseudo-triangulation $\calT$ then some
stabilization of $G$ is associated to the second barycentric
subdivision of $\calT$.
\end{proof}


We now describe the $1/4$ and $2/3$ bistellar flips in dimension
three.  These are also often called {\em Pachner moves}.  In any
triangulation, the $1/4$ flip replaces one tetrahedron by four; add a
vertex at the center of the chosen tetrahedron and cone to the faces.
Similarly the $2/3$ flip replaces a pair of distinct tetrahedra,
adjacent along a face, by three; remove the face, replace it by a dual
edge, and add three faces.  The $4/1$ and $3/2$ flips are the
reverses.  See \reffig{Flip} for illustrations of the $1/3$ and $2/2$
flips in dimension two.

\begin{figure}[ht!]
$$\begin{array}{c}
\psfig{file=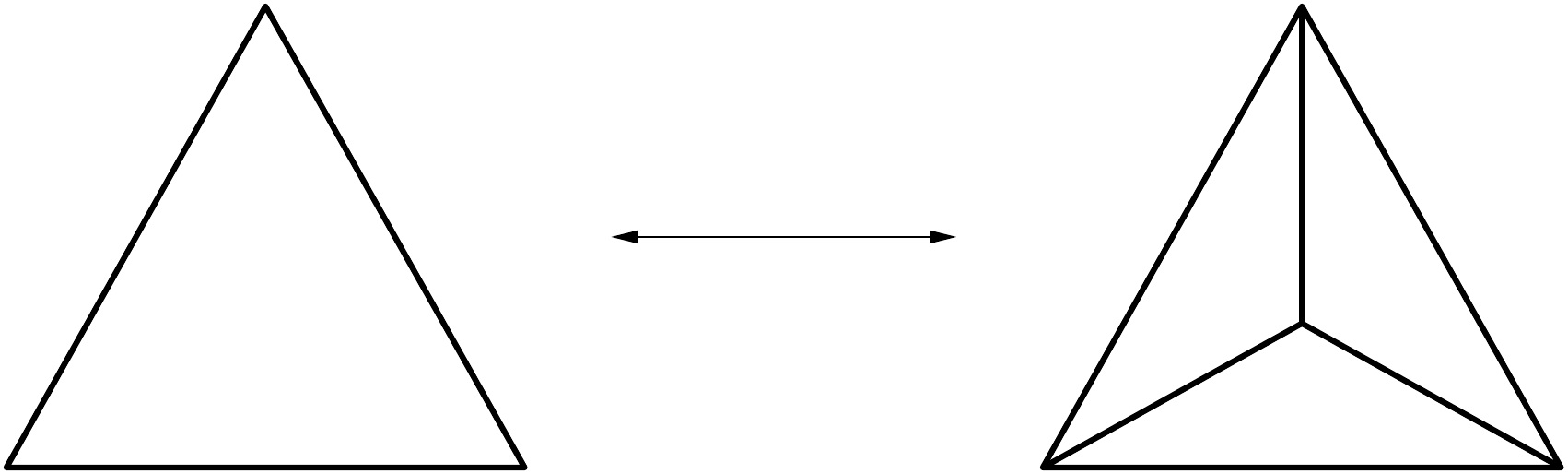, height = 2.5 cm} \\
\\
\\
\psfig{file=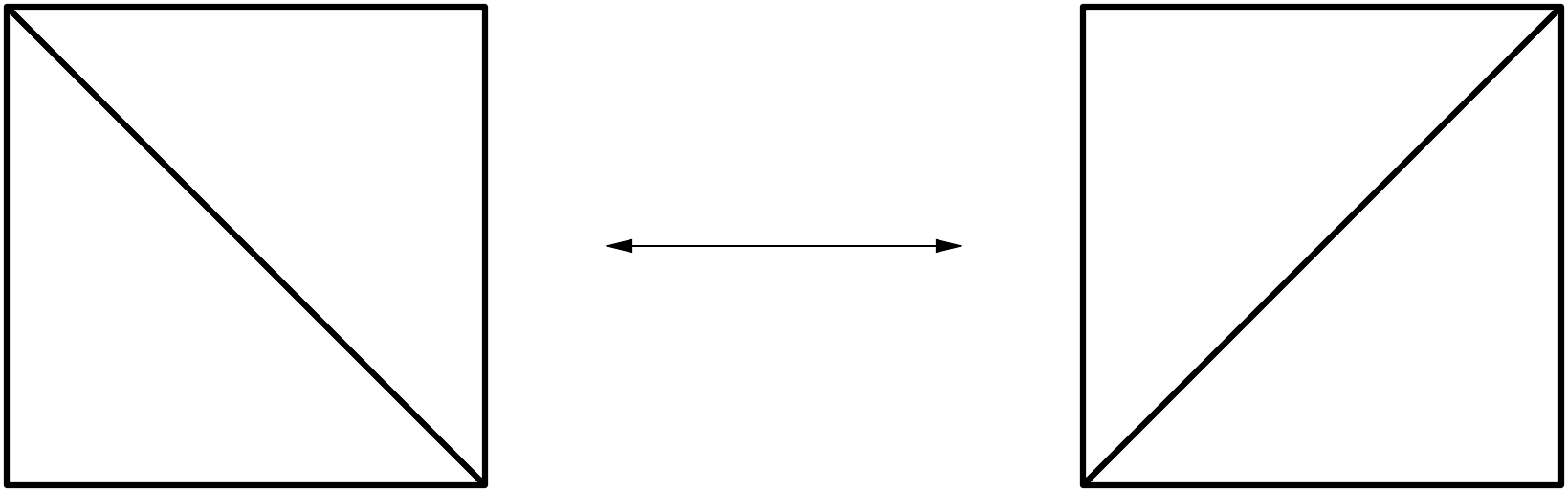, height = 2.5 cm}
\end{array}$$
\caption{The $1/3$ and $2/2$ bistellar flips}
\label{Fig:Flip}
\end{figure}

Suppose that $(M, F)$ and $(M, F')$ are associated to triangulations
$\calK$ and $\calK'$.  Now, if $\calK'$ is obtained from $\calK$
via a $2/3$ bistellar flip then $(M, F')$ is the stabilization of
$(M, F)$.  When a $1/4$ flip is used then $(M, F')$ is the third
stabilization of $(M, F)$. 

We may now state an important corollary of the Hauptvermutung
(\ref{Thm:Hauptvermutung}), due to Pachner~\cite{Pachner91}.

\begin{theorem}
\label{Thm:Pachner}
Suppose that $M$ is a closed three-manifold and $\calK, \calK'$ are
triangulations of $M$.  Then there is a sequence of isotopies and
bistellar flips that transforms $\calK$ into $\calK'$. \qed
\end{theorem}


Lickorish's article~\cite{Lickorish99} gives a discussion of Pachner's
Theorem and its application to the construction of three-manifold
invariants.  Now we have:

\begin{proof}[Proof of \refthm{RS}]
Suppose that $(M, F)$ and $(M, F')$ are a pair of splittings.  Using
\reflem{ExistsATriangulation} stabilize each to obtain splittings,
again called $F$ and $F'$, which are associated to triangulations.  By
Pachner's Theorem (\ref{Thm:Pachner}) these triangulations are related
by a sequence of bistellar flips and isotopy.  Consecutive splittings
along the sequence are related by stabilization or destabilization.
The uniqueness of stabilization now implies that $(M, F)$ and $(M,
F')$ are stably equivalent.
\end{proof}

\section{Meridian disks}
\label{Sec:Meridian}

We carefully study meridian disks of handlebodies before diving into
the proof proper of Waldhausen's Theorem (\ref{Thm:Waldhausen}).

\subsection*{Meridianal pairs} 

If $V$ is a handlebody and $v_0 \subset V$ is a properly embedded
disk, with $\bdy v_0$ essential in $\bdy V$, then we call $v_0$ a {\em
meridianal} disk of $V$.  Fix now a splitting $(M, F)$.  Let $V$ and
$W$ be the handlebodies that are the closures of the components of $M
\setminus F$.  So $V \cap W = F$.

\begin{definition}
\label{Def:DestabilizingPair}
Suppose that $v_0$ and $w_0$ are meridianal disks of $V$ and $W$.  Suppose
that $\bdy v_0$ and $\bdy w_0$ meet exactly once, transversely.  Then we
call $\{v_0, w_0\}$ a {\em meridianal pair} for $(M, F)$.
\end{definition}

Note that $\{v_0, w_0\}$ is often called a {\em destabilizing pair}.
To explain this terminology, one must check that $V' = V \setminus
U(v_0)$ and $W' = W \cup \overline{U(v_0)}$ are both handlebodies.
Thus, taking $F' = \bdy V' = \bdy W'$, we find that $(M, F')$ is a
Heegaard splitting and that $(M, F) \equiv (M, F') \connect (S^3, T)$.  

\begin{remark}
\label{Rem:AmbientIsotopic}
If $\{v_1, w_1\}, \ldots, \{v_n, w_n\}$ are pairwise disjoint
meridianal pairs then $V' = V \setminus U(\cup_i v_i)$ is ambient
isotopic to $V'' = V \cup U(\cup_i w_i)$.  When $n = 1$ this is a
pleasant exercise and the general case then follows from disjointness.
\end{remark}

Furthermore, in this situation $V' = V \setminus U(\cup_i v_i)$ and
$W' = W \cup \overline{U(\cup_i v_i)}$ are handlebodies.  So $F' =
\bdy V' = \bdy W'$ gives a splitting $(M, F')$ and we have $(M, F) =
(M, F') \connect_n (S^3, T)$.

Conversely, fix a splitting equivalent to $(M, F) \connect_n (S^3,
T)$.  There is a natural choice of pairwise disjoint meridianal pairs
$\{v_1, w_1\}, \ldots, \{v_n, w_n\}$ so that the above construction
recovers $(M, F)$.  As we shall see, the choice of pairs is not
unique.  This leads to the non-uniqueness of destabilization.

Suppose now that we have two splittings $(M, F)$ and $(M, G)$ that we
must show are equivalent.  By the Reidemeister--Singer Theorem above we
may stabilize to obtain equivalent splittings $(M, F') \equiv (M,
G')$.  So $(M, F')$ admits two collections of pairwise disjoint
meridianal pairs.  These record the handles of $F'$ that must be cut
to recover $F$ or $G$.  If, under suitable conditions, we can make our
collections similar enough then we can deduce that the original
splittings $(M, F)$ and $(M, G)$ are equivalent.  Unfortunately, our
process for modifying collections of meridianal pairs does not
preserve pairwise disjointness.  To deal with this Waldhausen
introduces the notions of {\em good} and {\em great} systems of
meridianal disks.

\subsection*{Good and great systems}

Fix a splitting $(M, F)$ with handlebodies $V$ and $W$.  Fix an
ordered collection $v = \{ v_1, \ldots, v_n \}$ of disjoint meridian
disks of $V$.

\begin{definition}
\label{Def:Good/Great}
We say $v$ is a {\em good system} if there is an ordered collection $w
= \{ w_1, \ldots, w_n \}$ of disjoint meridian disks of $W$ so that
\begin{itemize}
\item
$\{v_i, w_i\}$ is a meridianal pair for all $i$ and
\item
$v_i \cap w_j = \emptyset$ whenever $i > j$.  
\end{itemize}
If the latter condition holds whenever $i \neq j$ then we call
$v$ a {\em great system}.  In either case we call
$w$ a {\em $v$--determined} system.
\end{definition}


Both conditions can be understood via the intersection matrix $A =
|v_i \cap w_j|$.  For $v$ to be a good system we must find a system
$w$ so that $A$ is upper-triangular, with ones on the diagonal.  For
$v$ to be great $A$ must be the identity matrix.

\begin{lemma}[Waldhausen 2.2, part 1]
\label{Lem:GoodImpliesGreat}
Every good system is great.
\end{lemma}

\begin{proof}
Suppose that $v = \{v_1, \ldots, v_n\}$ is good and $w = \{w_1,
\ldots, w_n\}$ is the given $v$--determined system.  We may assume that
$w$ has been isotoped to minimize $|v \cap w|$.  If $v$ is also great
with respect to $w$ then we are done.

Supposing otherwise, let $k$ be the smallest index so that $v_k \cap
w$ is not a single point.  It follows that $v \cap w_k$ is a single
point.  Let $\alpha$ be a subarc of $\bdy v_k$ so that $\bdy \alpha$
is contained in $\bdy w$, one point of $\bdy \alpha$ lies in $\bdy
w_k$, and the interior of $\alpha$ is disjoint from $w$.

It follows that the other endpoint of $\alpha$ lies in $\bdy w_l$ for
some $l > k$.  Let $N = \overline{U(w_k \cup \alpha \cup w_l)}$ be a
closed regular neighborhood of the indicated union.  Then $\bdy N \cap
W$ consists of three essential disks, two of which are parallel to
$w_k$ and $w_l$.  Let $w'_l$ be the remaining disk.  Let $w' = (w
\setminus \{w_l\}) \cup \{w'_l\}$.  It follows that $v$ is still good
with respect to $w'$ and the total intersection number has been
decreased.  By induction, we are done.
\end{proof}

\begin{remark}
\label{Rem:Handleslide}
The last step of the proof may be phrased as follows: obtain a new
disk $w'_l$ via a {\em handle-slide} of $w_l$ over $w_k$ along the arc
$\alpha$.  The hypotheses tell us that the chosen slide does not
destroy ``goodness.''
\end{remark}


\begin{lemma}[Waldhausen 2.2, part 2]
\label{Lem:GoodImpliesIsotopic}
Suppose that $v$ is a good system with respect to $w$.  Then $V
\setminus U(v)$ and $V \cup \overline{U(w)}$ are ambient isotopic in
$M$.
\end{lemma}

\begin{proof}
By \fullref{Rem:AmbientIsotopic} the lemma holds when $w$ makes $v$
a great system.  Thus, by the proof of
\fullref{Lem:GoodImpliesGreat} all we need check is that $V \cup
\overline{U(w)}$ is isotopic to $V \cup \overline{U(w')}$, where $w$
and $w'$ are assumed to differ by a single handle-slide.  This
verification is an easy exercise.
\end{proof}


\subsection*{Reduction of $(M, F)$ by $v$}

Let $v$ be a good system with respect to $w$.  Since $v$ does not
separate $V$ the difference $V \setminus U(v)$ is a handlebody, as is
$W \setminus U(w)$.
By \reflem{GoodImpliesIsotopic} the unions $V \cup \overline{U(w)}$
and $W \cup \overline{U(v)}$ are also handlebodies.  Let $F(v)$ be the
boundary of $V \setminus U(v)$.  It follows that $(M, F(v))$ is a
Heegaard splitting.  We will call this the {\em reduction} of $(M, F)$
along $v$.  Taking $F(w)$ equal to the boundary of $W \setminus U(w)$
we likewise find that $(M, F(w))$ is a splitting.  With the induced
orientations, we find that $(M, F(v)) \equiv (M, F(w))$.  We
immediately deduce:

\begin{lemma}[Waldhausen 2.4]
\label{Lem:Transitivity}
If $v$ and $v'$ are both good systems with respect to $w$ then $(M,
F(v)) \equiv (M, F(v'))$.  \qed
\end{lemma}

From the Reidemeister--Singer Theorem and the definitions we have:
 
\begin{lemma}[Waldhausen 2.5, part 1]
\label{Lem:StableSytems}
Suppose that $(M, F_1)$ and $(M, F_2)$ have a common stabilization
$(M, F)$.  Then there is a system $v \subset V$ good with respect to
$w \subset W$ and a system $x \subset V$ good with respect to $y
\subset W$ so that $(M, F(v)) \equiv (M, F_1)$ and $(M, F(x)) \equiv (M,
F_2)$. \qed
\end{lemma}

\begin{remark}
\label{Rem:StableSystems}
We now have one decomposition and two sets of instructions for
reducing (cutting open trivial handles).  If we knew, for example,
that $y$ was a $v$--determined system then we would be done; but this
is more than we actually need.
\end{remark}

\subsection*{Getting along with your neighbors}

\begin{lemma}[Waldhausen 2.5, part 2]
\label{Lem:EmptyIntersection}
In the preceding lemma, $F$, $v$, $x$, $w$, $y$ can be chosen so that
$v \cap x = w \cap y = \emptyset$.
\end{lemma}

\begin{proof}
We proceed in several steps.

\medskip
{\bf Step 1}\qua Apply a small isotopy to ensure:
\begin{itemize}
\item
$(x \cap y) \cap (v \cup w) = \emptyset = (v \cap w) \cap (x \cup y)$.
\item
$v \cap x$ and $w \cap y$ are collections of pairwise disjoint simple
  closed curves and arcs. 
\item
$v \cap x \cap F = \bdy(v \cap x)$ and $w \cap y \cap F = \bdy(w \cap
  y)$. 
\end{itemize}

\medskip
{\bf Step 2}\qua Now we eliminate all simple closed curves of
intersection between $v$ and $x$.  Suppose that $v \cap x$ contains a
simple closed curve.  Then there is an {\em innermost} disk $D \subset
v$ so that $D \cap x = \bdy D$.  Use $D$ to perform a {\em disk
surgery} on $x$: since $x$ is a union of disks, $\bdy D$ bounds a
disk, say $D' \subset x$.  Let $x'$ be a copy of $(x \setminus D')
\cup D$, after a small isotopy supported in $U(D)$.  Arrange matters
so that $|v \cap x'| \leq |v \cap x|$.  By \reflem{Transitivity}, $(M,
F(x)) \equiv (M, F(x'))$.  Proceeding in this fashion, remove all
simple closed curves of $v \cap x$.  Apply the same procedure to
remove all simple closed curves of $w \cap y$.

\medskip{\bf Step 3}\qua Now we eliminate all arcs of intersection between $v$
and $x$.  To do this, we will replace $F$, and the various systems, by
highly stabilized versions.  Let $k$ be an arc of $v_i \cap x_j$.  Let
$v_i'$ and $v_i''$ be the two components of $v_i \setminus U(k)$.
These are both disks.  Similarly, let $x_j'$ and $x_j''$ be the two
components of $x_j \setminus U(k)$.  Choose notation so that $|v_i'
\cap w_i| = 1$, $|v_i'' \cap w_i| = 0$, and similarly for $x_j'$ and
$x_j''$.  Let $\overline{w}$ and $\overline{y}$ be disjoint spanning
disks of the cylinder $U(k) \cap V$.  Take $F' = \bdy (V \setminus
U(k))$.

Observe that
\begin{itemize}
\item
$(M, F')$ is a Heegaard splitting and is a stabilization of $(M, F)$.
\item
The system 
$$
v' = \{ v_1, \ldots, v_{i-1}, v_i', v_i'', v_{i+1}, \ldots, v_n \}
$$ 
is good with respect to the system 
$$
w' = \{ w_1, \ldots, w_i, \overline{w}, w_{i+1}, \ldots, w_n \}.
$$
\item
The same holds for $x'$ and $y'$.
\item
$(M, F'(w')) \equiv (M, F(w))$ and $(M, F'(y')) \equiv (M, F(y))$.
\item
$|v' \cap x'| < |v \cap x|$ and $|w' \cap y'| = |w \cap y|$.
\end{itemize}
Repeated stabilization in this fashion removes all arcs of
intersection and so proves the lemma.
\end{proof}

\section{The proof of Waldhausen's Theorem}
\label{Sec:Proof}

We may now begin the proof of: 

\begin{theorem}[Waldhausen 3.1]
\label{Thm:Waldhausen}
Suppose that $(S^3, G)$ is an unstabilized Heegaard splitting.  Then
$(S^3, G) \equiv (S^3, S^2)$.
\end{theorem}

This, and the uniqueness of stabilization, immediately implies our
earlier version of the theorem: up to isotopy, the three-sphere has a
unique splitting of every genus.

Let $(S^3, G)$ be an unstabilized splitting.  By
Lemmas~\ref{Lem:StableSytems} and~\ref{Lem:EmptyIntersection} there is
a splitting $(S^3, F)$ that is a common stabilization of $(S^3, G)$
and $(S^3, S^2)$ with several useful properties.  First, let $V, W$
denote handlebodies so that $V \cup W = S^3$, $V \cap W = F$.  Next,
note that $\genus(F) \geq \genus(G)$.  Letting $n = \genus(F)$ and $m
= \genus(F) - \genus(G)$ we assume that;
\begin{itemize}
\item
There are good systems $v = \{ v_1, \ldots, v_n \}$ and $x = \{ x_1,
\ldots, x_m \}$ in $V$.
\item
There is a $v$--determined system $w = \{ w_1, \ldots, w_n \}$ and an
$x$--determined system $y = \{ y_1, \ldots, y_m \}$ in $W$.  
\item
$(S^3, S^2) \equiv (S^3, F(v))$ and $(S^3, G) \equiv (S^3, F(x))$.
\item
$x \cap v = \emptyset = y \cap w$.  
\end{itemize}

Suppose that the surface $F$ is also chosen with minimal possible
genus.  We shall show, via contradiction, that $\genus(F) = 0$.  Since
$F$ was a stabilization of $G$ it will follow that $\genus(G) = 0$, as
desired.   So assume for the remainder of the proof that $n > 0$.

\begin{lemma}[Waldhausen 3.2]
Altering $y$ only we can ensure that $|y \cap v_n| \leq 1$.  
\end{lemma}


\begin{proof}
There are two possible cases.

\medskip{\bf Case 1}\qua Suppose some element of $y$ hits $v_n$ in at least two
points. Let $C = W \setminus U(w)$.  (This is a three-ball with {\em
spots}.)  Note that $y$ is a collection of disjoint disks in $C$.
Thus the disks $y$ cut $C$ into a collection of three-balls.  Note
that $w \cap \bdy v_n$ is a single point.  Hence $\gamma = \bdy v_n
\cap \bdy C$ is a single arc with interior disjoint from the spots of
$\bdy C$.  Since some element of $y$ hits $\bdy v_n$ twice there is an
element $y_j \in y$ and a subarc $\alpha$ contained in the interior of
$\gamma$ so that $\alpha \cap U(w) = \emptyset$, $\bdy \alpha \subset
y_j$, and $\interior(\alpha) \cap y = \emptyset$.

Choose an arc $\beta$, properly embedded in $y_j$, so that $\bdy \beta
= \bdy \alpha$.  Then $\alpha \cup \beta$ bounds a disk $D \subset C$
so that $D \cap \bdy C = \alpha$ and $D \cap y = \beta$.  Again, this
is true because $C \setminus y$ is a collection of three-balls.  (The
disk $D$ is called a {\em bigon}.)  Let $E$ be the component of $y_j
\setminus \beta$ that meets $x_j$ exactly once.  Let $y_j' = D \cup
E$.  (The modern language is that $y_j'$ is obtained from $y_j$ via
{\em bigon surgery} along $D$.)

Since $v \cap x = \emptyset$ it follows that $\alpha \cap x_j =
\emptyset$.  Thus $y_j'$ meets $x_j$ exactly once, $x_i \cap y_j' =
\emptyset$ for all $i > j$, and $y_i \cap y_j' = \emptyset$ for all $i
\neq j$.  Thus $y' = (y \setminus \{ y_j \}) \cup \{ y_j' \}$ is an
$x$--determined system.  Furthermore $y' \cap w = \emptyset$ and $y'$
meets $v_n$ fewer times than $y$ does.

\medskip
{\bf Case 2}\qua Suppose every disk in $y$ meets $v_n$ in at most one
point, and $|y \cap v_n| \geq 2$.  Define $C = W \setminus U(w)$ as
above.  There is an arc $\alpha \subset (\bdy v_n) \cap \bdy C$ so
that $\alpha \cap y = \bdy \alpha$.  We may assume that one point of
$\bdy \alpha$ lies in $y_i$ while the other lies in $y_j$, for $i <
j$.  Let $y_j'$ be the disk obtained by doing a handle-slide of $y_j$
over $y_i$ along the arc $\alpha$.  As indicated in
\refrem{Handleslide}, the system $y' = (y \setminus \{ y_j \}) \cup \{
y_j' \}$ has all of the desired properties, and also reduces
intersection with $v_n$.

Finally, iterating Case 1 and then Case 2 proves the lemma.
\end{proof}

\begin{proof}[Proof of \refthm{Waldhausen}](Waldhausen 3.3)

{\bf Case 1}\qua If $y \cap v_n \neq \emptyset$ then by the above lemma
we can assume that $y \cap v_n$ is a single point.  Suppose that $y_j$
meets $v_n$.

Define
$$
x' = \{ x_1, \ldots, x_{j-1}, x_{j+1}, \ldots, x_m, v_n \},
$$
$$
y' = \{ y_1, \ldots, y_{j-1}, y_{j+1}, \ldots, y_m, y_j \},
$$ 
and notice that $x'$ is good with respect to $y'$.
\reflem{Transitivity} implies that $(S^3, F(y')) \equiv (S^3, F(x'))$
and $(S^3, F(y)) \equiv (S^3, G)$.  Since $y$ and $y'$ are equal as
sets $(S^3, F(y')) \equiv (S^3, F(y))$.  So $(S^3, F(x')) \equiv (S^3,
G)$.

Now we replace $y'$ by another $x'$--determined system $y''$ by
replacing $y_m'$ by $w_n$.  That is, define
$$y'' = \{ y_1', \ldots, y_{m-1}', w_n \}.$$ The meridianal pair
$(v_n, w_n) = (x_m', y_m'')$ represents the first trivial handle cut
off in the process of transforming $(S^3, F)$ into $(S^3, F(v))$ or
$(S^3, F(x'))$.  So the first step in the process of transforming
$(S^3, F)$ into $(S^3, F(x')) \equiv (S^3, G)$ is the same as the
first step in going from $(S^3, F)$ to $(S^3, F(v)) \equiv (S^3,
S^2)$.  Let $(S^3, F')$ be the Heegaard decomposition obtained from
$(S^3, F)$ by cutting off this trivial handle. Then $(S^3, F')$ has
the same properties as $(S^3, F)$ but $F'$ has lower genus than $F$.
This contradicts the minimality of $F$.

\medskip
{\bf Case 2}\qua If $y \cap v_n = \emptyset$ then we enlarge $x$ and $y$
to $x^*$ and $y^*$ by adding $v_n$ and $w_n$.  That is, we define
$$x^* = \{ x_1, \ldots, x_m, v_n \}$$
$$y^* = \{ y_1, \ldots, y_m, w_n \}.\leqno{\hbox{and}}$$ 
Suppose in $(S^3, F)$ we cut
off the trivial handles of $(x^*, y^*)$, obtaining $(S^3, F(x^*))$.
Then we effectively cut off all the trivial handles of $(x,y)$,
obtaining $(S^3, F(x)) \equiv (S^3, G)$ and additionally cut off the
trivial handle represented by $(v_n, w_n)$.  

So $(S^3, F(x^*))$ is obtained from $(S^3, G)$ by removing a trivial
handle.  That is, $(S^3, G) \equiv (S^3, F(x)) \equiv (S^3, F(x^*))
\connect (S^3, T)$.  Thus $G$ is a stabilized splitting.  This is a
contradiction.
\end{proof}

\section{Remarks}
\label{Sec:Remarks}

\subsection*{Doubling a handlebody}
Suppose that $T \subset S^2 \cross S^1$ is the torus obtained by
taking the product of the equator of the two-sphere and the $S^1$
factor.  Let $(M_g, F_g) = \connect_g (S^2 \cross S^1, T)$.  Notice
that $M_g$ may also be obtained by forming doubling a genus $g$
handlebody across its boundary.

Waldhausen appears to claim the following:

\begin{theorem}[Waldhausen 4.1]
\label{Thm:Double}
$F_g$ is the unique unstabilized splitting of $M_g$, up to isotopy.
\end{theorem}

\noindent
His actual sentence is: 
\begin{quote}
Hieraus und aus [\refthm{Waldhausen}] folgt, da\ss{} auch die
Mannigfaltigkeiten [$M_g$] nur die bekannten Heegaard-Zerlegungen
besitzen.
\end{quote}
(Brackets added.)  This indicates that \refthm{Double} follows from
Haken's Lemma~\cite{Haken68} and \refthm{Waldhausen}.  It is clear
that Haken's Lemma can be used to prove that $F_g$ is unique up to
homeomorphism.  It is not clear to this writer how to obtain
\refthm{Double} by following Waldhausen's remark.

It seems that no proof of \refthm{Double} appears in the literature
until the recent work of Carvalho and Oertel on automorphisms of
handlebodies.  See Theorem~1.10 of their
paper~\cite{CarvalhoOertel05}.  A similar proof may be given using
Hatcher's normal form for sphere systems (Proposition~1.1
of~\cite{Hatcher95}).  Carvalho and Oertel also give an alternative
proof, deducing \refthm{Double} from work of
Laudenbach~\cite{Laudenbach73}.

\subsection*{Compression bodies}

\begin{definition}[Waldhausen 4.2]
Suppose that $V$ is a handlebody and $D$ is a (perhaps empty) system
of meridianal disks properly embedded in $V$.  Let $N$ be a closed
regular neighborhood of $D \cup \bdy V$, taken in $V$.  Then $N$ is a
{\em compression body}.
\end{definition}

Note that $\bdy N$ is disconnected and contains $\bdy V$ as a
component.  This component is called the {\em positive} boundary of
$N$ and is denoted by $\bdy_+ N$.  The {\em negative} boundary is
$\bdy_- N = \bdy N \setminus \bdy_+ N$.  Most modern authors disallow
copies of $S^2$ appearing in $\bdy_- N$.

Now suppose $M$ is an orientable three-manifold and $F \subset M$ is a
orientable closed surface in the interior of $M$.  If $F$ cuts $M$
into two pieces $V$ and $W$, where each of $V$ and $W$ is a handlebody
or a compression body, and where $F = \bdy_+ V = \bdy_+ W$ then we say
that $(M, F)$ is a Heegaard splitting of $M$ with respect to the {\em
partition} $(V \cap \bdy M, W \cap \bdy M)$.  Equivalence (up to
isotopy), stabilization, and stable equivalence with respect to a
fixed partition may all be defined as above.  The Reidemeister--Singer
Theorem can then be extended: any two Heegaard splittings of $M$
giving the same partition of $\bdy M$ are stably equivalent.


\subsection*{Haken's Lemma in compression bodies}

Haken's Lemma also applies to Heegaard splittings respecting a
partition.  Similarly, suppose that $(M, F)$ is a Heegaard splitting
respecting a partition and $D \subset M$ is a properly embedded disk
so that $\bdy D$ is essential in $\bdy M$.  Then there is another such
disk meeting $F$ is a single curve.  Using this and
\refthm{Waldhausen} we have:

\begin{theorem}[Waldhausen 4.3]
If $V$ is a handlebody and $(V, F)$ is an unstabilized splitting then
$F$ is parallel to $\bdy V$.  \qed
\end{theorem}


\subsection*{Lens spaces}

As noted above, in addition to equivalence up to isotopy, we may
define another equivalence relation on splittings $(M, F)$; namely
equivalence up to orientation-preserving homeomorphism of pairs.  If
$\bdy M \neq \emptyset$ then we also require that the partition of
$\bdy M$ be respected.  Notice that these two equivalence relations do
not generally agree, for example in the presence of {\em
incompressible} tori.  For a modern discussion, with references,
see Bachman and Derby-Talbot~\cite{BachmanDerby06}.


Waldhausen notes that connect sum makes either set of equivalence
classes into a commutative and associative monoid.  This monoid is
{\em not} cancellative.  Suppose that $(M,F)$ is a genus one splitting
of a lens space, not equal to the three-sphere.  Then $(M, F)$ is
characterized, up to homeomorphism, by a pair of relatively prime
integers $(p, q)$ with $0 < q < p$.  Now, letting $-F$ represent $F$
with the opposite orientation, we find that $(M, -F)$ is characterized
by $(p, q')$ where
$$
q \cdot q' = 1 \left(\mbox{mod $p$}\right).
$$ 
It follows that $(M, F)$ and $(M, -F)$ are generally not equivalent.
On the other hand, $(M, F) \connect (S^3, T)$ and $(M, -F) \connect
(S^3, T)$ are always equivalent.  For suppose that $D \subset F$ is a
small disk, $N$ is a closed regular neighborhood of $F \setminus
\interior(D)$, and $G = \bdy N$.  Then $G$ is the desired common
stabilization.

Waldhausen ends by suggesting that the pairs $(M,F)$ characterized by
$(5,2)$ and $(7,2)$, and their orientation reverses (namely $(5,3)$
and $(7,4)$), have interesting connect sums.  He wonders how many
distinct equivalence classes, up to isotopy or up to homeomorphism,
are represented by the four sums
$$
(5,2) \connect (7,2),\ (5,2) \connect (7,4),\ (5,3) \connect (7,2),\
(5,3) \connect (7,4).
$$ 
This question was answered by Engmann~\cite{Engmann70}; no pair of the
suggested genus two splittings are homeomorphic.

\bibliographystyle{gtart}
\bibliography{link}

\begin{thebibliography}{}
\providecommand\bibmarginpar{\leavevmode\marginpar}
\def\urlstyle#1{{\tt #1}}

\bibitem{Alexander24}
\textbf{J\,W Alexander}, \emph{{On the subdivision of 3-space by a
  polyhedron.}}, Nat. Acad. Proc. 10 (1924) 6--8

\bibitem{BachmanDerby06}
\textbf{D Bachman}, \textbf{R Derby-Talbot},
  \href{http://dx.doi.org/10.2140/agt.2006.6.351} {\emph{Non-isotopic
  {H}eegaard splittings of {S}eifert fibered spaces {\rm , with an appendix by
  R Weidmann}}}, Algebr. Geom. Topol. 6 (2006) 351--372 \xox{MR}{2220681}

\bibitem{BartoliniRubinstein06}
\textbf{L Bartolini}, \textbf{J\,H Rubinstein},
  \href{http://dx.doi.org/10.2140/agt.2006.6.1319} {\emph{One-sided {H}eegaard
  splittings of {$\Bbb R{\rm P}\sp 3$}}}, Algebr. Geom. Topol. 6 (2006)
  1319--1330 \xox{MR}{2253448}

\bibitem{Bing59}
\textbf{R\,H Bing}, \href{http://dx.doi.org/10.2307/1970092} {\emph{An
  alternative proof that {$3$}-manifolds can be triangulated}}, Ann. of Math.
  $(2)$ 69 (1959) 37--65 \xox{MR}{0100841}

\bibitem{CarvalhoOertel05}
\textbf{L\,N Carvalho}, \textbf{U Oertel}, \emph{A classification of
  automorphisms of compact $3$--manifolds} \xox{arXiv}{math/0510610}

\bibitem{CassonGordon85}
\textbf{A\,J Casson}, \textbf{C\,M Gordon}, \emph{Manifolds with irreducible
  {H}eegaard splittings of arbitrarily high genus}, talk at MSRI (1985)

\bibitem{Craggs76}
\textbf{R Craggs}, \href{http://dx.doi.org/10.2307/2040883} {\emph{A new proof
  of the {R}eidemeister-{S}inger theorem on stable equivalence of {H}eegaard
  splittings}}, Proc. Amer. Math. Soc. 57 (1976) 143--147 \xox{MR}{0410749}

\bibitem{Engmann70}
\textbf{R Engmann}, \emph{Nicht-hom\"oomorphe {H}eegaard-{Z}erlegungen vom
  {G}eschlecht {$2$} der zusammenh\"angenden {S}umme zweier {L}insenr\"aume},
  Abh. Math. Sem. Univ. Hamburg 35 (1970) 33--38 \xox{MR}{0283803}

\bibitem{FomenkoMatveev97}
\textbf{A\,T Fomenko}, \textbf{S\,V Matveev}, \emph{Algorithmic and computer
  methods for three-manifolds}, Mathematics and its Applications 425, Kluwer
  Academic Publishers, Dordrecht (1997) \xox{MR}{1486574}\ \ Translated from
  the 1991 Russian original by M Tsaplina and M Hazewinkel and revised by the
  authors

\bibitem{Gugenheim53}
\textbf{V\,K A\,M Gugenheim}, \emph{Piecewise linear isotopy and embedding of
  elements and spheres. {I}, {II}}, Proc. London Math. Soc. $(3)$ 3 (1953)
  29--53, 129--152 \xox{MR}{0058204}

\bibitem{Haken68}
\textbf{W Haken}, \emph{Some results on surfaces in {$3$}-manifolds}, from:
  ``Studies in Modern Topology'', Math. Assoc. Amer. (distributed by
  Prentice-Hall) (1968)  39--98 \xox{MR}{0224071}

\bibitem{Hamilton76}
\textbf{A\,J\,S Hamilton}, \emph{The triangulation of {$3$}-manifolds}, Quart.
  J. Math. Oxford Ser. $(2)$ 27 (1976) 63--70 \xox{MR}{0407848}

\bibitem{Hatcher95}
\textbf{A Hatcher}, \href{http://dx.doi.org/10.1007/BF02565999}
  {\emph{Homological stability for automorphism groups of free groups}},
  Comment. Math. Helv. 70 (1995) 39--62 \xox{MR}{1314940}

\bibitem{Hatcher01}
\textbf{A Hatcher}, \emph{Notes on basic 3-manifold topology} (2001)
\ Available at \setbox0\hbox{\makeatletter\@url
{http://www.math.cornell.edu/~hatcher/3M/3Mdownloads.html}}
\href{http://www.math.cornell.edu/~hatcher/3M/3Mdownloads.html}
{\unhbox0}

\bibitem{Hempel76}
\textbf{J Hempel}, \emph{{$3$}-{M}anifolds}, Ann. of Math. Studies 86,
  Princeton University Press (1976) \xox{MR}{0415619}

\bibitem{JacoRubinstein06}
\textbf{W Jaco}, \textbf{J\,H Rubinstein}, \emph{{Layered-triangulations of
  3-manifolds}} \xox{arXiv}{math/0603601}

\bibitem{Johnson07}
\textbf{J Johnson}, \emph{Notes on {H}eegaard splittings} (2007)
\ Available at \setbox0\hbox{\makeatletter\@url
{http://www.math.ucdavis.edu/~jjohnson/notes.pdf}}
\href{http://www.math.ucdavis.edu/~jjohnson/notes.pdf}
{\unhbox0}

\bibitem{Kobayashi92}
\textbf{T Kobayashi}, \emph{A construction of {$3$}-manifolds whose
  homeomorphism classes of {H}eegaard splittings have polynomial growth}, Osaka
  J. Math. 29 (1992) 653--674 \xox{MR}{1192734}

\bibitem{Laudenbach73}
\textbf{F Laudenbach}, \href{http://dx.doi.org/10.2307/1970877} {\emph{Sur les
  {$2$}-sph\`eres d'une vari\'et\'e de dimension {$3$}}}, Ann. of Math. $(2)$
  97 (1973) 57--81 \xox{MR}{0314054}

\bibitem{Lei00}
\textbf{F Lei}, \href{http://dx.doi.org/10.1017/S0305004100004461} {\emph{On
  stability of {H}eegaard splittings}}, Math. Proc. Cambridge Philos. Soc. 129
  (2000) 55--57 \xox{MR}{1757777}

\bibitem{Lickorish99}
\textbf{W\,B\,R Lickorish}, \emph{Simplicial moves on complexes and manifolds},
  from: ``Proceedings of the Kirbyfest (Berkeley, CA, 1998)'', (J Hass, M
  Scharlemann, editors), Geom. Topol. Monogr. 2 (1999)  299--320
  \xox{MR}{1734414}

\bibitem{LustigMoriah00}
\textbf{M Lustig}, \textbf{Y Moriah},
  \href{http://dx.doi.org/10.1016/S0040-9383(99)00024-5} {\emph{3-manifolds
  with irreducible {H}eegaard splittings of high genus}}, Topology 39 (2000)
  589--618 \xox{MR}{1746911}

\bibitem{Moise52}
\textbf{E\,E Moise}, \href{http://dx.doi.org/10.2307/1969769} {\emph{Affine
  structures in {$3$}-manifolds. {V}. {T}he triangulation theorem and
  {H}auptvermutung}}, Ann. of Math. $(2)$ 56 (1952) 96--114 \xox{MR}{0048805}

\bibitem{Moise77}
\textbf{E\,E Moise}, \emph{Geometric topology in dimensions {$2$} and {$3$}},
  Graduate Texts in Mathematics 47, Springer, New York (1977) \xox{MR}{0488059}

\bibitem{MoriahEtAl06}
\textbf{Y Moriah}, \textbf{S Schleimer}, \textbf{E Sedgwick},
  \href{http://projecteuclid.org/getRecord?id=euclid.cag/1175790071}
  {\emph{Heegaard splittings of the form {$H+nK$}}}, Comm. Anal. Geom. 14
  (2006) 215--247 \xox{MR}{2255010}

\bibitem{Pachner91}
\textbf{U Pachner}, \emph{P.{L}. homeomorphic manifolds are equivalent by
  elementary shellings}, European J. Combin. 12 (1991) 129--145
  \xox{MR}{1095161}

\bibitem{Reidemeister33}
\textbf{K Reidemeister}, \emph{Zur dreidimensionalen Topologie.}, Abh. Math.
  Semin. Hamb. Univ. 9 (1933) 189--194

\bibitem{Rieck06}
\textbf{Y Rieck}, \emph{{A proof of Waldhausen's uniqueness of splittings of
  $S^3$ (after Rubinstein and Scharlemann)}} \xox{arXiv}{math/0607332}

\bibitem{Rolfsen76}
\textbf{D Rolfsen}, \emph{Knots and links}, Mathematics Lecture Series 7,
  Publish or Perish, Berkeley, CA (1976) \xox{MR}{0515288}

\bibitem{RourkeSanderson72}
\textbf{C\,P Rourke}, \textbf{B\,J Sanderson}, \emph{Introduction to
  piecewise-linear topology}, Ergebnisse series 69, Springer, New York (1972)
  \xox{MR}{0350744}

\bibitem{RubinsteinScharlemann96}
\textbf{H Rubinstein}, \textbf{M Scharlemann},
  \href{http://dx.doi.org/10.1016/0040-9383(95)00055-0} {\emph{Comparing
  {H}eegaard splittings of non-{H}aken {$3$}-manifolds}}, Topology 35 (1996)
  1005--1026 \xox{MR}{1404921}

\bibitem{Scharlemann02}
\textbf{M Scharlemann}, \emph{Heegaard splittings of compact 3-manifolds},
  from: ``Handbook of geometric topology'', North-Holland, Amsterdam (2002)
  921--953 \xox{MR}{1886684}

\bibitem{ScharlemannThompson94a}
\textbf{M Scharlemann}, \textbf{A Thompson}, \emph{Thin position and {H}eegaard
  splittings of the {$3$}-sphere}, J. Differential Geom. 39 (1994) 343--357
  \xox{MR}{1267894}

\bibitem{Sedgwick99}
\textbf{E Sedgwick}, \emph{The irreducibility of {H}eegaard splittings of
  {S}eifert fibered spaces}, Pacific J. Math. 190 (1999) 173--199
  \xox{MR}{1722770}

\bibitem{Singer33}
\textbf{J Singer}, \href{http://dx.doi.org/10.2307/1989314}
  {\emph{Three-dimensional manifolds and their {H}eegaard diagrams}}, Trans.
  Amer. Math. Soc. 35 (1933) 88--111 \xox{MR}{1501673}

\bibitem{Waldhausen68}
\textbf{F Waldhausen}, \href{http://dx.doi.org/10.1016/0040-9383(68)90027-X}
  {\emph{Heegaard-{Z}erlegungen der {$3$}-{S}ph\"are}}, Topology 7 (1968)
  195--203 \xox{MR}{0227992}

\end{thebibliography}
\end{document}